\documentclass[11pt]{amsart}
\usepackage{amsmath}
\usepackage{amssymb}
\usepackage{amscd}
\def\NZQ{\mathbb}               % the font for N,Z,Q,R,C
\def\NN{{\NZQ N}}

\def\ZZ{{\NZQ Z}}
\def\RR{{\NZQ R}}

\def\PP{{\NZQ P}}

\newtheorem{Theorem}{Theorem}[section]
\newtheorem{Lemma}[Theorem]{Lemma}
\newtheorem{Corollary}[Theorem]{Corollary}
\newtheorem{Proposition}[Theorem]{Proposition}

\newtheorem{Example}[Theorem]{Example}

\let\epsilon\varepsilon
\let\phi=\varphi
\let\kappa=\varkappa

\begin{document}
\title{The asymptotic  growth of graded linear series on arbitrary projective schemes}
\author{Steven Dale Cutkosky }
\thanks{Partially supported by NSF}

\address{Steven Dale Cutkosky, Department of Mathematics,
University of Missouri, Columbia, MO 65211, USA}
\email{cutkoskys@missouri.edu}

\begin{abstract} Recently,  Okounkov \cite{Ok1}, Lazarsfeld and Mustata \cite{LM}, and Kaveh and Khovanskii \cite{KK} have shown that the growth of a  graded linear series
on a projective variety over an algebraically closed field is asymptotic to a polynomial. We give a complete description of the possible asymptotic growth of
graded linear series on projective schemes over a perfect field. If the scheme is reduced, then the growth is polynomial like, but the growth can be very complex
on nonreduced schemes. 

We also give an example of a graded family of $m$-primary ideals $\{I_n\}$ in a nonreduced $d$-dimensional local ring $R$, such that the length of $R/I_n$ divided by $n^d$  does not have a limit, even when restricted to any  arithmetic sequence.
\end{abstract}

\maketitle
\section{Introduction} 
 In this paper we investigate the asymptotic growth of a graded linear series on an arbitrary projective scheme over a perfect field.
Recent results of Okounkov \cite{Ok1}, Lazarsfeld and Mustata \cite{LM}, and Kaveh and Khovanskii \cite{KK} show the remarkable fact that on a projective variety  over an algebraically closed field $k$, the dimension $\dim_kL_n$ of a graded linear series $L$ is asymptotic for large $n$ to the value of a polynomial; that is,
$$
\lim_{n\rightarrow \infty}\frac{\dim_kL_{mn}}{n^q}
$$
exists, where $m$ is the index of $L$ and $q\ge 0$ is the Kodiara-Iitaka dimension of $L$.
We recall  this result in Theorem \ref{Theorem10} below, which is the  general statement of Kaveh and Khovanskii \cite{KK}. The limit can be irrational, even when $L$ is the section ring of a big line bundle,
as is shown by Srinivas and the author  in Example 4 of Section 7 \cite{CS}.

It is  natural to consider the question of the existence of such limits when we loosen  these conditions. 
We give a complete description of how much of  Theorem \ref{Theorem10} extends and how much does not extend to arbitrary projective schemes over a perfect field. In Theorem \ref{Theorem13} we show that the exact statement of Theorem \ref{Theorem10} holds for graded linear series on a projective variety over a perfect field. In Theorem \ref{Theorem18} we show that the theorem generalizes very well to graded linear series on a reduced projective variety over a perfect field,
although the statement requires a slight modification. The conclusion is that there is a positive integer $r$ such that for any integer $a$, the limit 
$$
\lim_{n\rightarrow \infty}\frac{\dim_kL_n}{n^q}
$$
exists whenever $n$ is constrained to line in the arithmetic sequence $a+br$. Here $q$ is the Kodaira-Iitaka dimension of $L$. A nontrivial example
on a connected, reduced, equidimensional  but not irreducible projective scheme is given in Example \ref{Example3}.

This is however the extent to which Theorem \ref{Theorem10} generalizes. We give a series of examples of graded linear series on non reduced projective schemes
where such limits do not exist. 

The most significant example is Example \ref{Example25}, which is of
a ``big'' graded linear system $L$ with maximal Kodaira-Iitaka dimension $d$ on a nonreduced but irreducible $d$-dimensional projective scheme such that the limit 
$$
\lim_{n\rightarrow \infty}\frac{\dim_kL_n}{n^d}
$$
does not exist, even when $n$ is constrained to lie in {\it any} arithmetic sequence. 

We give a related example, Example \ref{Example1}, showing the  failure of limits of lengths of quotients of a graded family of $m_R$-primary ideals $\{I_n\}$ in the $d$-dimensional local ring $R=k[[x_1,\ldots,x_d,y]]/y^2$. In
the example,
the limit 
\begin{equation}\label{eqI7}
\lim_{n\rightarrow \infty}\frac{\ell_R(R/I_n)}{n^d}
\end{equation}
does not exist, even when $n$ is constrained to lie in {\it any} arithmetic sequence. 

Such limits for graded families of ideals in local rings are shown to exist in many cases by work of Ein, Lazarsfeld and Smith \cite{ELS}, Mustata \cite{Mus}, Lazarsfeld and Mustata \cite{LM} and of the author \cite{C}. In Theorem 5.9 of \cite{C}, we show that the limit (\ref{eqI7}) always exists when $\{I_n\}$ is a graded family of  $m_R$-primary ideals in a $d$-dimensional  analytically unramified  equicharacteristic local ring with perfect residue field, and give a number of applications of this result. 
It follows from Example \ref{Example1} that the assumption of analytically unramified cannot be removed from  Theorem 5.9 \cite{C}. 

A fundamental property of the Kodaira-Iitaka dimension $q(L)$ for a graded linear series $L$ on a reduced projective scheme with $q(L)\ge 0$ is that there exists a 
constant $\beta$ such that there is an upper bound
$$
\dim_kL_n<\beta n^{q(L)}
$$
for all $n$. 
This is a classical result for complete linear systems on varieties, and is part of the foundations of the Kodaira-Iitaka  dimension (Theorem 10.2 \cite{I}). A proof of this inequality for graded linear series on reduced projective varieties over perfect fields is given in this paper in Corollary \ref{Cor72}. However, this equality may fail on non reduced projective schemes. We always have an upper bound
$$
\dim_kL_n<\gamma n^{d}
$$
where $d$ is the dimension of the scheme (\ref{eqKI4}), but it is possible on a $d$-dimensional nonreduced projective scheme to have $q(L)=-\infty$ and  have growth of  order  $n^d$.
We give a simple example where this happens in Example \ref{Example2}. In fact, it is quite easy to construct badly behaved examples with $q(L)=-\infty$, since in this case the condition that $L_mL_n\subset L_{m+n}$ required for a graded linear series may be  vacuous. 

We give in Example \ref{Example16} an example of a section ring of a line bundle $\mathcal N$ on a non reduced but irreducible $d$-dimensional projective scheme $Z$ with growth of 
order $n^{d-1}$ such that
for any  positive integer $r$,  there exists an integer $a$ such that the limit 
$$
\lim_{n\rightarrow\infty}\frac{\dim_k\Gamma(Z,\mathcal N^n)}{n^{d-1}}
$$
does not exist when $n$ is constrained to lie in the arithmetic sequence $a+br$.

Even on nonreduced projective schemes, we do have the classical property of Kodaira-Iitaka dimension that if $q(L)\ge 0$, then there is a positive constant $\alpha$ and a positive integer $m$ such that
$$
\alpha n^{q(L)}<\dim_kL_{mn}
$$
for all integers $n$ (\ref{eqKI2}).

The volume of a line bundle $\mathcal L$ on a $d$-dimensional variety $X$ is the limsup
\begin{equation}\label{in2}
\mbox{Vol}(\mathcal L)=\limsup_{n\rightarrow\infty}\frac{h^0(X,\mathcal L^n)}{n^d/d!}.
\end{equation}
There has  been spectacular progress of our understanding of  the volume as a function on the big cone in $N^1(X)$ on a projective variety $X$
over an algebraically closed field (where (\ref{in2}) is actually a limit). Much of the  theory  is explained in \cite{La}, where extensive  references are given. Volume is continuous  on $N^1(X)$ but is not twice differentiable on all of $N^1(X)$ (as shown in an example of Ein Lazarsfeld, Mustata, Nakamaye and Popa, \cite{ELMNP1}).   Bouksom, Favre and Jonsson \cite{BFJ} have shown that the volume is $\mathcal C^1$-differentiable on the big cone of $N^1(X)$. Interpretation of the directional derivative in terms of intersection products and many interesting applications are given in \cite{BFJ}, \cite{ELMNP1} and \cite{LM}.

The starting point of the theory of volume on nonreduced schemes is to determine if the limsup defined in (\ref{in2}) exists as a limit. We see from Examples \ref{Example25} and \ref{Example2}  that the limit does  not always exist for graded linear series $L$. However, neither of these examples are section rings of a line bundle. The examples are on
the nonreduced scheme $X$ which is  a double linear hyperplane in a projective space $\PP^{d+1}$. All line bundles on $X$ are restrictions of line bundles on $\PP^{d+1}$, so that if $\mathcal L$ is a line bundle on $X$, then $h^0(X,\mathcal L^n)$ is actually a polynomial in $n$, and $\mbox{Vol}(\mathcal L)$ not only exists as a limit, it is even a rational number.

We essentially use the notation of Hartshorne \cite{H}. For instance, a variety is required to be integral. We will denote the maximal ideal of a local ring $R$ by $m_R$. If $\nu$ is a valuation of a field $K$, then we will write $V_{\nu}$ for the valuation ring of $\nu$, and $m_{\nu}$ for the maximal ideal of $V_{\nu}$. We will write $\Gamma_{\nu}$ for the value group of $\nu$. If $A$ and $B$ are local rings, we will say that $B$ dominates $A$ if $A\subset B$ and $m_B\cap A=m_A$.

\section{Graded linear series and the Kodaira-Iitaka dimension}\label{Section2} 

Suppose that $X$ is a $d$-dimensional projective scheme over a field $k$, and $\mathcal L$ is a line bundle on $X$. 
Then under the natural inclusion of rings $k\subset \Gamma(X,\mathcal O_X)$, we have that the section ring
$$
\bigoplus_{n\ge 0}\Gamma(X,\mathcal L^n)
$$
is a graded $k$-algebra. A graded $k$-subalgebra $L=\bigoplus_{n\ge 0}L_n$ of a section ring of a line bundle $\mathcal L$ on $X$ is called a {\it graded linear series} for $\mathcal L$. 

 We define the {\it Kodaira-Iitaka dimension} $q=q(L)$ of a graded linear series $L$  as follows.
Let
$$
\sigma(L)=\max \left\{m\mid 
\begin{array}{l}
 \mbox{there exists $y_1,\ldots,y_m\in L$ which are homogeneous of  positive}\\
\mbox{degree and are algebraically independent over $k$}
\end{array}\right\}.
$$
$q(L)$ is then defined as
$$
q(L)=\left\{\begin{array}{ll}
\sigma(L)-1 &\mbox{ if }\sigma(L)>0\\
-\infty&\mbox{ if }\sigma(L)=0
\end{array}\right.
$$

This definition is in agreement with the classical definition for line bundles on projective varieties (Definition in Section 10.1 \cite{I}).
We give a summary of a few formulas which hold for the Kodaira-Iitaka dimension on general projective schemes. We defer proofs of these formulas to the appendix at the end of this paper. 

\begin{Lemma}\label{LemmaKI} Suppose that $L$ is a graded linear series on a $d$-dimensional projective scheme $X$ over a field $k$. Then
\begin{enumerate}\item[1.]
\begin{equation}\label{eqKI1}
q(L)\le d=\dim X.
\end{equation}
\item[2.] There exists a positive constant $\gamma$ such that 
\begin{equation}\label{eqKI4}
\dim_k L_n<\gamma n^d
\end{equation}
for all $n$.
\item[3.] Suppose that $q(L)\ge 0$. Then there exists a positive constant $\alpha$ and a positive integer $e$ such that 
\begin{equation}\label{eqKI2}
\dim_kL_{en}>\alpha n^{q(L)}
\end{equation}
for all positive integers $n$. 
\item[4.] Suppose that $X$ is reduced and $L$ is a graded linear series on $X$. Then $q(L)=-\infty$ if and only if $L_n=0$ for all $n>0$.

\end{enumerate}
\end{Lemma}

\section{Some remarkable limits}
 
Suppose that $L$ is a graded linear series on a projective variety $X$. 
The {\it index} $m=m(L)$ of $L$ is defined as the index of groups
$$
m=[\ZZ:G]
$$
where $G$ is the subgroup of $\ZZ$ generated by $\{n\mid L_n\ne 0\}$.

\begin{Theorem}\label{Theorem10} (Okounkov  \cite{Ok1}, Lazarsfeld and Mustata \cite{LM}, Kaveh and Khovanskii \cite{KK}) Suppose that $X$ is a projective variety over an algebraically closed field $k$, and $L$ is a graded linear series on $X$. Let $m=m(L)$ be the index of $L$ and $q=q(L)\ge 0$ be the Kodaira-Iitaka dimension of $L$.  Then  
$$
\lim_{n\rightarrow \infty}\frac{\dim_k L_{nm}}{n^q}
$$
exists. 
\end{Theorem}

In particular, from the definition of the index, we have that the limit
$$
\lim_{n\rightarrow \infty}\frac{\dim_k L_{n}}{{n}^q}
$$
exists, whenever $n$ is constrained to lie in an arithmetic sequence $a+bm$ ($m=m(L)$ and $a$ an arbitrary but fixed constant).

It follows that $\dim_kL_n=0$ if $m\not\,\mid n$, and if $q(L)\ge 0$, then there exist positive constants $\alpha<\beta$ such that
\begin{equation}\label{eq61}
\alpha n^q<\dim_kL_{nm}<\beta n^q
\end{equation}
for all sufficiently large positive integers $n$

The proof of the theorem is by an ingenious method, reducing the problem to computing the volume of a section (the Newton-Okounkov body) of an appropriate cone.

\begin{Corollary}\label{Theorem12}(Okounkov  \cite{Ok1}, Lazarsfeld and Mustata \cite{LM}) Suppose that $X$ is a projective variety of dimension $d$ over an algebraically closed field $k$, and $\mathcal L$ is a big line bundle on $X$ (the Kodaira-Iitaka dimension of the section ring of $\mathcal L$ is $d$).  Then  the limit
$$
\lim_{n\rightarrow \infty}\frac{\dim_k \Gamma(X,\mathcal L^n)}{n^d}
$$
exists. 
\end{Corollary}

This corollary was earlier proven using Fujita approximation, \cite{Fuj}, Example 11.4.7 \cite{La}, \cite{T}.

An example of a big line bundle where the limit in Theorem \ref{Theorem10} and Corollary \ref{Theorem12} is an irrational number is given in Example 4 of Section 7 \cite{CS}.

\section{Limits on varieties over non closed fields}

 The proof of Koveh and Khovanskii of Theorem \ref{Theorem10} actually shows the following. 
\begin{Theorem}(Koveh and Khovanskii)\label{Theorem53}
Suppose that $X$ is a $d$-dimensional projective variety over an arbitrary field $k$ and there exists a valuation $\nu$ of the function field  $k(X)$ of $X$ which vanishes on nonzero elements of $k$, has a value group $\Gamma_{\nu}$ which is isomorphic as a group to $\ZZ^d$, and the residue field of $V_{\nu}/m_{\nu}$ is $k$.
Then the conclusions of Theorem \ref{Theorem10} are valid for any graded linear series $L$ on $X$ with $q(L)\ge 0$.
\end{Theorem}

These conditions on the valuation are stated before Definition 2.26 \cite{KK}. The condition that $\nu$ has ``one dimensional leaves'' is defined before Proposition 2.6 \cite{KK}. It is equivalent to the condition that the residue field of the valuation is $k$.
 
\begin{proof} Suppose that such a valuation $\nu$ exists. Let $V_{\nu}$ be the valuation ring of $\nu$ in the function field $k(X)$.
Let $Q$ be the center of $\nu$ on $X$; that is, $\mathcal O_{X,Q}$ is the local ring of $X$ which is dominated by $V_{\nu}$. $Q$ exists and is unique since a projective variety is proper. $Q$ is a $k$-rational point on $X$ since the residue field of $V_{\nu}$ is $k$.

$L$ is a graded linear series for some line bundle $\mathcal L$ on $X$. Let $m=m(L)$ and $q=q(L)$. Since $X$ is integral, $\mathcal L$ is isomorphic to an invertible sheaf
$\mathcal O_X(D)$ for some Cartier divisor $D$ on $X$. 
We can assume that $Q$ is not contained in the support of $D$, after possibly replacing $D$ with a Cartier divisor linearly equivalent to $D$.
We have an induced graded $k$-algebra isomorphism of section rings 
$$
\bigoplus_{n\ge 0}\Gamma(X,\mathcal L^n)\rightarrow \bigoplus_{n\ge 0}\Gamma(X,\mathcal O_X(nD))
$$
which takes $L$ to a graded linear series for $\mathcal O_X(D)$. Thus we may assume that $\mathcal L=\mathcal O_X(D)$.
 For all $n$, the restriction map followed by inclusion into $V_{\nu}$,
\begin{equation}\label{eqR3}
\Gamma(X,\mathcal L^n)\rightarrow \mathcal L_Q=\mathcal O_{X,Q}\subset V_{\nu}
\end{equation}
 is a 1-1 $k$-vector space homomorphism since $X$ is integral, and we have an induced $k$-algebra homomorphism 
$$
 L\rightarrow \mathcal O_{X,Q}\subset V_{\nu}.
 $$

Given a nonnegative element $\gamma$ in the  value group $\Gamma_{\nu}$ of $\nu$, which is isomorphic to $\ZZ^d$ as a group, with some total ordering,  we have associated valuation ideals $I_{\gamma}$ and $I_{\gamma}^+$ in $V_{\nu}$ defined by 
$$
I_{\gamma}=\{f\in V_{\nu}\mid \nu(f)\ge \gamma\}
$$
and
$$
I_{\gamma}^+=\{f\in V_{\nu}\mid \nu(f)>\gamma\}.
$$
Since $V_{\nu}/m_{\nu}=k$,  we have the critical condition that 
\begin{equation}\label{eqR1}
\dim_k I_{\gamma}/I_{\gamma}^+=1
\end{equation}
for all non negative $\gamma \in \Gamma_{\nu}$.
Let
$$
S_n=\{\gamma\in \Gamma_{\nu}\mid \mbox{ there exists } f\in L_n\mbox{ such that }\nu(f)=\gamma\}.
$$
By (\ref{eqR1}) and  (\ref{eqR3}), we that
$$
\dim_k L_n\cap I_{\gamma}/L_n\cap I_{\gamma}^+=\left\{\begin{array}{ll}
1&\mbox{ if there exists $f\in L_n$ with $\nu(f)=\gamma$}\\
0&\mbox{ otherwise.}
\end{array}
\right.
$$
Since every element of $L_n$ has non negative value (as $L_n\subset V_{\nu}$), we have that 
\begin{equation}\label{eqR2}
\dim_kL_n=|S_n|
\end{equation}
for all $n$.
Let 
$$
S(L)=\{(\gamma,n)|\gamma \in S_n\},
$$
a subsemigroup of $\ZZ^{d+1}$.
By Theorem 2.30 \cite{KK}, $S(L)$ is a ``strongly nonnegative semigroup'', so by Theorem 1.26 \cite{KK} and (\ref{eqR2}), we have that
$$
\lim_{n\rightarrow \infty}\frac{\dim_kL_{nm}}{n^q}=\lim_{n\rightarrow \infty}\frac{|S_{mn}|}{n^q}
$$
exists, where $m=m(L)$, and is proportional to the volume of the Newton-Okounkov body $\Delta(S(L))=\pi^{-1}(m)\cap\mbox{Con}(S(L))$, where $\mbox{Con}(S(L))$ is the closure of the  cone generated by $S(L)$ in $\RR^{d+1}$ and $\pi:\RR^{d+1}\rightarrow \RR$ is the projection onto the last factor.

\end{proof}

The condition that there exists a valuation as in the assumptions of Theorem \ref{Theorem53} is always satisfied if $k$ is algebraically closed. It is however a rather special condition over non closed fields, as is shown by the following proposition.

\begin{Proposition}\label{Prop1} Suppose that $X$ is a $d$-dimensional projective variety over a field $k$. Then there exists a valuation $\nu$ of the function field $k(X)$ of $X$ 
such that the value group $\Gamma_{\nu}$ of $\nu$ is isomorphic to $\ZZ^d$  and the residue field $V_{\nu}/m_{\nu}=k$ if and only if 
there exists a birational morphism $X'\rightarrow X$ of projective varieties such that there exists a nonsingular (regular) $k$-rational point $Q'\in X'$.
\end{Proposition} 

\begin{proof}First suppose  there exists a valuation $\nu$ of the function field $k(X)$ of $X$ 
such that the value group $\Gamma_{\nu}$ of $\nu$ is isomorphic to $\ZZ^d$ as a group and with residue field $V_{\nu}/m_{\nu}=k$. Then $\nu$ is an ``Abhyankar valuation''; that is 
$$
\mbox{trdeg}_kk(X)=d=0+d= \mbox{trdeg}_kV_{\nu}/m_{\nu}+\mbox{rational rank }\Gamma_{\nu},
$$
with $k=V_{\nu}/m_{\nu}$, so there exists a local uniformization of $\nu$ by \cite{KKu}. Let $Q$ be the center of $\nu$ on $X$, so that $V_{\nu}$ dominates $\mathcal O_{X,Q}$. $\mathcal O_{X,Q}$ is a localization of a $k$-algebra $k[Z]$ where $Z\subset V_{\nu}$ is a finite set. By Theorem 1.1 \cite{KKu}, there exists a regular local ring $R$ which is essentially of finite type over $k$ with quotient field $k(X)$ such that $V_{\nu}$ dominates $R$ and $Z\subset R$. Since $k[Z]\subset R$ and $V_{\nu}$ dominates $\mathcal O_{X,Q}$, we have that $R$ dominates $\mathcal O_{X,Q}$. The residue field $R/m_R=k$ since $V_{\nu}$ dominates $R$.
There exists a projective $k$-variety $X''$ such that
$R$ is the local ring of a closed $k$-rational point $Q'$ on $X''$, and the birational map $X''\dashrightarrow X$ is a morphism in a neighborhood of $Q'$. Let $X'$ be the graph of the birational correspondence between $X''$ and $X$.
Since $X''\dashrightarrow X$ is a morphism in a neighborhood of $Q'$, the projection of $X'$ onto $X''$ is an isomorphism in a neighborhood of $Q'$. We can thus
identify $Q'$ with a nonsingular $k$-rational point of $X'$.

Now suppose that there exists a birational morphism $X'\rightarrow X$ of projective varieties such that there exists a nonsingular  $k$-rational point $Q'\in X'$. 

Choose a regular system of parameters $y_1,\ldots, y_d$ in $R=\mathcal O_{X',Q'}$. $R/m_R=k(Q')=k$, so $k$ is  a coefficient field of $R$.
 We have that $\hat R=k[[y_1,\ldots,y_d]]$. We define a valuation $\hat \nu$ dominating $\hat R$ by stipulating that
 \begin{equation}\label{eq20}
\hat\nu(y_i)= e_i\mbox{ for $1\le i\le d$}
\end{equation}
where $\{e_i\}$ is the standard basis of the totally ordered  group $(\ZZ^d)_{\rm lex}$, and
 $\hat\nu(c)=0$ if $c$ is a nonzero element of $k$.
 
If $f\in \hat R$ and $f=\sum c_{i_1,\ldots,i_d}y_1^{i_1}\cdots y_d^{i_d}$ with $c_{i_1,\ldots,i_d}\in k$, then 
$$
\hat \nu(f)=\min\{\nu(y_1^{i_1}\cdots y_d^{i_d})\mid c_{i_1,\ldots,i_d}\ne0\}.
$$
We let $\nu$ be the valuation of the function field $k(X)$ which is obtained by restricting $\nu$. The value group of $\nu$ is $(\ZZ^d)_{\rm lex}$.

Suppose that $h$ is in $k(X)$ and $\nu(h)=0$. Write $h=\frac{f}{g}$ where $f,g\in R$ and $\nu(f)=\nu(g)$. Thus in $\hat R$,
we have expansions $f=\alpha y_1^{i_1}\cdots y_d^{i_d} +f'$, $g=\beta y_1^{i_1}\cdots y_d^{i_d} +g'$ where $\alpha,\beta$ are nonzero elements of $k$, $\nu(y_1^{i_1}\cdots y_d^{i_d})=\nu(f)=\nu(g)$
and $\nu(f')>\nu(f)$, $\nu(g')>\nu(g)$. Let  $\gamma=\frac{\alpha}{\beta}$ in $k$. Computing $f-\gamma g$ in $\hat R$, we obtain that $\nu(f-\gamma g)>\nu(f)$, and thus the residue of $\frac{f}{g}$ in $V_{\nu}/m_{\nu}$ is equal to the residue of $\gamma$, which is in $k$. By our construction $k\subset V_{\nu}$. Thus the residue field $V_{\nu}/m_{\nu}=k$.
\end{proof}

\begin{Corollary}\label{Corollary1} Suppose that $X$ is a projective variety over a field $k$ which has  a nonsingular $k$-rational point. Then the conclusions of Theorem 
\ref{Theorem10} hold for any graded linear series $L$ on $X$.
\end{Corollary}

\begin{proof} This is immediate from Theorem \ref{Theorem53} and  Proposition \ref{Prop1}.
\end{proof}

We obtain  the following extension of  Theorem \ref{Theorem10}.

\begin{Theorem}\label{Theorem13} Suppose that $X$ is a projective variety over a perfect field $k$. 

Let $L$ be a graded linear series on $X$. Let $m=m(L)$ be the index of $L$ and $q=q(L)\ge 0$ be the Kodaira-Iitaka dimension of $L$.  Then  
$$
\lim_{n\rightarrow \infty}\frac{\dim_k L_{nm}}{n^q}
$$
exists. 
\end{Theorem}

\begin{proof}  
Let $Q$ be a closed regular  point in $X$. Let $R=\mathcal O_{X,Q}$.
 Let   $k'$ be a Galois closure of the residue field $k(Q)$ of $R$ over $k$.
$k'$ is finite  separable over $k$, so that $X'=X\times_kk'$ is  reduced   as $\mathcal O_X\otimes_kk'$ is a subsheaf of rings of $k(X)\otimes_kk'$, which is reduced by Theorem 39, Section 15, Chapter III \cite{ZS1}.

 Let $S=R\otimes_kk'$. Then $S$ is a reduced  semi local ring by Theorem 39 \cite{ZS1}.  Let $p_1,\ldots, p_r$ be the maximal ideals of $S$. 
 $S/m_RS\cong (R/m_R)\otimes_kk'$ is  reduced by Theorem 39, \cite{ZS1}. Thus
$m_RS=p_1\cap\cdots \cap p_r$.
 Since $R$ is a regular local ring, $m_R$ is generated by $d=\dim R$ elements. For $1\le i\le r$, we thus have that $p_iS_{p_i}=m_RS_{p_i}$ is generated by $d=\dim R=\dim S_{p_i}$ elements. Thus $S_{p_i}$ is a regular local ring for all $i$, so $S$ is a regular ring.
 
  $k'=k[\alpha]$ for some $\alpha\in k'$ since $k'$ is a finite separable extension of $k$.
  Let $f(x)\in k[x]$ be the minimal polynomial of $\alpha$.  $k'$ is a separable normal extension of $k$ containing $\alpha$, so $f(x)$ splits into distinct linear factors in $k'[x]$. Then
$$
\bigoplus_{i=1}^r S/p_i\cong S/m_RS\cong k'\otimes_kk'\cong k'[x]/(f(x))\cong (k')^r.
$$
Thus  $S/p_i\cong k'$ for all $i$. 
Let $Q'_i$ be the corresponding closed point to $p_i$ in $X'$, which has the local ring $\mathcal O_{X',Q'_i}=S_{p_i}$, so that
$Q'_i$ is a regular, $k'$-rational point on the variety $X'$ for all $i$.

Let $X_1,\ldots, X_s$ be the distinct irreducible components of $X'$.
Since $X'$ is reduced, we have a natural inclusion
$$
0\rightarrow \mathcal O_{X'}\rightarrow \bigoplus_{i=1}^s\mathcal O_{X_i}
$$
which induces inclusions 
\begin{equation}\label{eq30}
\Gamma(X',\mathcal L^n\otimes_{\mathcal O_X}\mathcal O_{X'})\rightarrow \bigoplus_{i=1}^s\Gamma(X_i,\mathcal L^n\otimes_{\mathcal O_X}\mathcal O_{X_i})
\end{equation}
for all $n$.

The elements of the Galois group $G$ of $k'$ over $k$ induce $X$-automorphisms of $X'$ which act transitively on the components $X_i$. $G$ acts naturally on $\mathcal L^n\otimes_{\mathcal O_X}\mathcal O_{X'}$. Thus for $\sigma\in G$, we have a commutative diagram
$$
\begin{array}{ccccc}
L_n\otimes_kk'&\subset&\Gamma(X', \mathcal L^n\otimes_{\mathcal O_X}\mathcal O_{X'})&\rightarrow &\Gamma(X_1, \mathcal L^n\otimes_{\mathcal O_{X}}\mathcal O_{X_1})\\
\downarrow \mbox{id}&&\downarrow\sigma&&\downarrow\sigma\\
L_n\otimes_kk'&\subset&\Gamma(X', \mathcal L^n\otimes_{\mathcal O_X}\mathcal O_{X'})&\rightarrow &\Gamma(\sigma(X_1), \mathcal L^n\otimes_{\mathcal O_{X}}\mathcal O_{\sigma(X_1)}).
\end{array}
$$
Suppose that $h\in L_n\otimes_kk'$ maps to zero in  $\Gamma(X_1, \mathcal L^n\otimes_{\mathcal O_X}\mathcal O_{X_1})$.
Since $G$ acts transitively on the components of $X'$,  $h$ maps to zero in $\Gamma(X_i, \mathcal L^n\otimes_{\mathcal O_X}\mathcal O_{X_i})$ for all $i$. From the inclusion (\ref{eq30}), we conclude that $h=0$.  Thus we have  inclusions
$$
L_n\otimes_kk'\rightarrow \Gamma(X,\mathcal L^n\otimes_{\mathcal O_X}\mathcal O_{X_1})
$$
for all $n$.  

Let $L'=\bigoplus_{n\ge 0}L_n\otimes_kk'$. $L'$ is a graded linear series for the line bundle $\mathcal L\otimes_{\mathcal O_X}\mathcal O_{X_1}$ on the $k'$-variety $X_1$. We have that $m(L')=m(L)$ and $q(L')=q(L)$.

Since $Q_i\in X_1$ for some $i$,  we have that $X_1$ contains a non singular $k'$-rational point.
 By Corollary \ref{Corollary1}, the limit
$$
\lim_{n\rightarrow \infty} \frac{\dim_{k'} L'_{nm}}{n^q}
$$
thus exists.

 Now the theorem follows from the  formula 
$$
\dim_kL_n=\dim_{k'}L_n\otimes_kk'.
$$
\end{proof}

It follows from the theorem that (\ref{eq61}) holds for a graded linear series on a  projective variety over a perfect field $k$.

We obtain that  Corollary \ref{Theorem12} holds on reduced projective schemes over perfect fields.

\begin{Corollary}\label{Theorem14} Suppose that $X$ is a reduced projective scheme of dimension $d$ over a perfect field $k$, and $\mathcal L$ is a  line bundle on $X$.  Then  the limit
$$
\lim_{n\rightarrow \infty}\frac{\dim_k \Gamma(X,\mathcal L^n)}{n^d}
$$
exists. 
\end{Corollary}

\begin{proof} 
We first prove the corollary in the case when $X$ is integral (a variety). 
We may assume that the section ring $L$ of $\mathcal L$ has maximal Kodaira-Iitaka dimension $d$, because the limit is zero otherwise.
There then exists a positive constant $\alpha$ and a positive integer $e$ such that
$$
\dim_k\Gamma(X,\mathcal L^{ne})>\alpha n^d
$$
for all positive integers $n$ by (\ref{eqKI2}). Let $H$ be a hyperplane section of $X$, giving a short exact sequence
$$
0\rightarrow \mathcal O_X(-H)\rightarrow \mathcal O_X\rightarrow \mathcal O_H\rightarrow 0.
$$
Tensoring with $\mathcal L^n$ and taking global sections, we see that $\Gamma(X,\mathcal L^{ne}\otimes \mathcal O_X(-H))\ne 0$ for $n\gg 0$ as
$q(\mathcal L^e\otimes \mathcal O_H)\le \dim(H)=d-1$. Since $H$ is ample, there exists a positive integer $f$ such that $\mathcal L\otimes\mathcal O_X(fH)$
is generated by global sections. Thus
$$
\Gamma(X,\mathcal L^{nef+1})\cong \Gamma(X,(\mathcal L^{nef}\otimes \mathcal O_X(-fH))\otimes (\mathcal L\otimes \mathcal O_X(fH)))\ne 0
$$
for $n\gg 0$. Thus $m(L)=1$.
The corollary  in the case when $X$ is a variety thus follows from Theorem \ref{Theorem13}.

Now assume that $X$ is only reduced. Let $X_1,\ldots, X_s$ be the irreducible components of $X$. Since $X$ is reduced, we have a natural short exact sequence of
$\mathcal O_X$-modules
$$
0\rightarrow \mathcal O_X\rightarrow \bigoplus_{n\ge 0}\mathcal O_{X_i}\rightarrow \mathcal F\rightarrow 0
$$
where $\mathcal F$ has support of dimension $\le d-1$. Tensoring with $\mathcal L^n$, we obtain that
$$
\lim_{n\rightarrow \infty}\frac{\dim_k\Gamma(X,\mathcal L^n)}{n^d}=\sum_{i=1}^s\lim_{n\rightarrow \infty}\frac{\dim_k\Gamma(X_i, \mathcal L^n\otimes\mathcal O_{X_i})}{n^d}
$$
exists, as $\dim_k\Gamma(X,\mathcal F\otimes \mathcal L^n)$ grows at most like $n^{d-1}$.
\end{proof}

\section{Limits on Reduced Schemes}

Suppose that $X$ is a projective scheme over a field $k$ and $L$ is a graded linear series for a linebundle $\mathcal L$ on $X$.
Suppose that $Y$ is a closed subscheme of $X$. Set $\mathcal L|Y=\mathcal L\otimes_{\mathcal O_X}\mathcal O_Y$. Taking global sections of the natural surjections
$$
\mathcal L^n\stackrel{\phi_n}{\rightarrow} (\mathcal L|Y)^n\rightarrow 0,
$$
for $n\ge 1$ we have  induced short exact sequences of $k$-vector spaces
\begin{equation}\label{eq54}
0\rightarrow K(L,Y)_n\rightarrow L_n\rightarrow (L|Y)_n\rightarrow 0,
\end{equation}
where 
$$
(L|Y)_n:=\phi_n(L_n)\subset \Gamma(Y,({\mathcal L}|Y)^n)
$$
 and $K(L ,Y)_n$ is the kernel of $\phi_n|L_n$. Defining $K(L,U)_0=k$ and $(L|Y)+0=\phi_0(L_0)$, we have that
$L|Y=\bigoplus_{n\ge 0}(L|Y)_n$ is a graded linear series for $\mathcal L|Y$ and $K(L,Y)=\bigoplus_{n\ge 0}K(L,Y)_n$ is a graded linear series for $\mathcal L$.

\begin{Lemma}\label{Lemma50} Suppose that $X$ is a reduced projective scheme and $X_1,\ldots,X_s$ are the irreducible components of $X$. Suppose that $L$ is a graded linear series on $X$. Then 
$$
q(L)=\max\{q(L|X_i)\mid 1\le i\le s\}.
$$
\end{Lemma}

\begin{proof} $L$ is a graded linear series  for a line bundle $\mathcal L$ on $X$.
Let $X_1,\ldots,X_s$ be the irreducible components of $X$. Since $X$ is reduced, 
we have a natural inclusion
$$
0\rightarrow \mathcal O_X\rightarrow \bigoplus_{i=1}^s \mathcal O_{X_i}.
$$
There is a  natural inclusion of $k$-algebras
$$
\bigoplus_{n\ge 0}\Gamma(X,\mathcal L^n)\rightarrow  \bigoplus_{i=1}^s\left(\bigoplus_{n\ge 0}\Gamma(X_i,\mathcal L^n\otimes_{\mathcal O_X}\mathcal O_{X_i})\right),
$$ 
which induces an inclusion of $k$-algebras
\begin{equation}\label{eq42}
L\rightarrow \bigoplus_{i=1}^sL|X_i.
\end{equation}
Suppose that $i$ is such that $1\le i\le s$. Set $t=q(L|X_i)$. Then by the definition of Kodaira-Iitaka dimension, there exists a graded inclusion of $k$-algebras $\phi:k[z_1,\ldots,z_t]\rightarrow L|X_i$ where
$k[z_1,\ldots,z_t]$ is a graded polynomial ring. Since the projection $L\rightarrow L|X_i$ is a surjection, we have a lift of $\phi$ to a graded $k$-algebra homomorphism into $L$, which is 1-1, so that $q(L)\ge t$. Thus
$$
q(L)\ge \max\{q(L|X_i)\mid 1\le i\le s\}.
$$
Let $q=q(L)$. Then there exists a 1-1 $k$-algebra homomorphism $\phi:k[z_1,\ldots,z_q]\rightarrow L$ where
$k[z_1,\ldots,z_q]$ is a positively graded polynomial ring. Let $\phi_i:k[z_1,\ldots,z_q]\rightarrow L|X_i$ be the induced homomorphisms, for $1\le i\le s$.
Let $\mathfrak p_i$ be the kernel of $\phi_i$. Since (\ref{eq42}) is 1-1, we have that $\mathfrak p_1\cap \cdots \cap \mathfrak p_s=(0)$.
Since $k[z_1,\ldots,z_q]$ is a domain, this implies that some $\mathfrak p_i=(0)$. Thus $\phi_i$ is 1-1 and we have that $q(L|X_i)\ge q(L)$.

\end{proof}

\begin{Theorem}\label{Theorem18} Suppose that $X$ is a reduced projective scheme over a perfect field $k$. 
Let $L$ be a graded linear series on $X$. Let $q=q(L)\ge 0$ be the Kodaira-Iitaka dimension of $L$.  
 Then there exists a positive integer $r$ such that 
$$
\lim_{n\rightarrow \infty}\frac{\dim_k L_{a+nr}}{n^q}
$$
exists for any fixed $a\in \NN$.
\end{Theorem}
The theorem says that 
$$
\lim_{n\rightarrow \infty}\frac{\dim_k L_{n}}{n^q}
$$
exists if $n$ is constrained to lie in an arithmetic sequence $a+br$ with $r$ as above, and for some fixed $a$. The conclusions of the theorem are a little weaker than the conclusions of Theorem \ref{Theorem13} for integral varieties. In particular, the index $m(L)$ has no relevance on reduced but not irreducible varieties.

\begin{proof}  Let $X_1,\ldots,X_s$ be the irreducible components of $X$. 
 Define graded linear series $M^i$ on $X$
 by $M^0=L$, $M^i=K(M^{i-1},X_i)$ for $1\le i\le s$. 
 By (\ref{eq54}), for $n\ge 1$, we have exact sequence of $k$-vector spaces 
 $$
 0\rightarrow (M^{j+1})_n=K(M^j,X_{j+1})_n\rightarrow M_n^j\rightarrow (M^j|X_{j+1})_n\rightarrow 0
 $$
 for $1\le j\le s-1$, and
 $$
 M_n^j={\rm Kernel}(L_n\rightarrow \bigoplus_{i=1}^s(L|X_i)_n)
 $$
 for $0\le j\le s$. As in (\ref{eq42}) in the proof of Lemma \ref{Lemma50}, $L\rightarrow \bigoplus_{i=1}^sL|X_i$ is an injection of $k$-algebras since 
 $X$ is reduced. Thus $M_n^s=(0)$, and
 \begin{equation}\label{eq71}
 \dim_kL_n=\sum_{i=1}^{s}\dim_k (M^{i-1}|X_i)_n
 \end{equation}
 for all $n$.
  Let $r=\mbox{LCM}\{m(L|X_i)\mid  q(L|X_i)=q(L)\}$. 
 The theorem now follows from Theorem \ref{Theorem13} applied to each of the $X_i$ with $q(M^{i-1}|X_i)=q(L)$ (we can start with $X_1$ with $q(L|X_1)=Q(L)$).
\end{proof}

\begin{Corollary}\label{Cor72}
Suppose that $X$ is a reduced projective scheme over a perfect field $k$. 
Let $L$ be a graded linear series on $X$ with $q(L)\ge 0$. Then there exists a positive constant $\beta$ such that 
\begin{equation}\label{eq80}
\dim_kL_n<\beta n^{q(L)}
\end{equation}
for all $n$. Further, there exists a positive constant $\alpha$ and a positive integer $m$ such that
\begin{equation}\label{eq81}
\alpha n^{q(L)}<\dim_kL_{mn}
\end{equation}
for all positive integers $n$.
\end{Corollary}

\begin{proof} Equation (\ref{eq80}) follows from (\ref{eq71}), since $\dim_k(M^{i-1}|X_i)\le \dim_k(L|X_i)$ for all $i$, and since (\ref{eq61}) holds on a variety. Equation (\ref{eq81}) is immediate from (\ref{eqKI2}).
\end{proof}

The following lemma is required for the construction of the next example. It follows from Theorem V.2.17 \cite{H} when $r=1$.  The lemma uses the notation of \cite{H}.

\begin{Lemma}\label{Lemma52} Let $k$ be an algebraically closed field, and write $\PP^1=\PP^1_k$.
Suppose that $r\ge 0$. Let $X=\PP(\mathcal O_{\PP^1}(-1)\bigoplus \mathcal O_{\PP^1}^r)$ with natural projection $\pi:X\rightarrow \PP^1$. Then the complete linear
system $|\Gamma(X,\mathcal O_X(1)\otimes \pi^*\mathcal O_{\PP^1}(1))|$ is base point free, and the only curve contracted by the induced morphism of $X$ is the curve $C$ which is the section of $\pi$ defined by the projection of $\mathcal O(-1)_{\PP^1}\bigoplus \mathcal O_{\PP^1}^r$ onto the first factor.
\end{Lemma}

\begin{proof}
We prove this by induction on $r$. 

First suppose that $r=0$. Then $\pi$ is an isomorphism, and $X=C$.
$$
\mathcal O_X(1)\otimes \pi^*\mathcal O_{\PP^1}(1)\cong \pi_*\mathcal O_X(1)\otimes \mathcal O_{\PP^1}(1)\cong \mathcal O_{\PP^1}(-1)\otimes \mathcal O_{\PP^1}(1)\cong \mathcal O_{\PP^1},
$$
from which the statement of the lemma follows.

Now suppose that $r>0$ and the statement of the lemma is true for $r-1$. Let $V_0$ be the $\PP^1$-subbundle of $X$ corresponding to projection onto the first $r-1$
factors,
\begin{equation}\label{eq51}
 0\rightarrow \mathcal O_{\PP^1}\rightarrow \mathcal O_{\PP^1}(-1)\bigoplus \mathcal O_{\PP^1}^r \rightarrow \mathcal O_{\PP^1}(-1)\bigoplus \mathcal O_{\PP^1}^{r-1}\rightarrow 0.
 \end{equation}
 Apply $\pi_*$ to the exact sequence 
 $$
 0\rightarrow \mathcal O_X(1)\otimes \mathcal O_X(-V_0)\rightarrow \mathcal O_X(1)\rightarrow \mathcal O_{V_0}(1)\rightarrow 0
 $$
 to obtain the exact sequence (\ref{eq51}), from which we see that $\mathcal O_X(V_0)\cong \mathcal O_X(1)$ and $V_0\cong \PP(\mathcal O_{\PP^1}(-1)\bigoplus \mathcal O_{\PP^1}^{r-1})$ with $\mathcal O_X(V_0)\otimes \mathcal O_{V_0}\cong  \mathcal O_{V_0}(1)$. Let $F$ be the fiber over a point in $\PP^1$ by $\pi$. We have that $\mathcal O_X(1)\otimes \pi^*\mathcal O_{\PP^1}\cong \mathcal O_X(V_0+F)$.
 Apply $\pi_*$ to
 $$
 0\rightarrow \mathcal O_X(F) \rightarrow \mathcal O_X(V_0+F)\rightarrow \mathcal O_X(V_0+F)\otimes \mathcal O_{V_0}\rightarrow 0
 $$
 to get
 $$
 0\rightarrow \mathcal O_{\PP^1}(1)\rightarrow \mathcal O_{\PP^1}\bigoplus \mathcal O_{\PP^1}(1)^r\rightarrow \pi_*(\mathcal O_{V_0}(1)\otimes \pi^*\mathcal
 O_{\PP^1}(1))\rightarrow 0.
 $$
 Now take global sections to obtain that the restriction map 
 $$
 \Gamma(X,\mathcal O_X(V_0+F))\rightarrow \Gamma(V_0, \mathcal O_{V_0}(1)\otimes \pi^*\mathcal
 O_{\PP^1}(1)) 
 $$
 is a surjection. In particular, by the induction statement, $V_0$ contains no base points of $\Lambda=|\Gamma(X,\mathcal O_X(V_0+F))|$. 
 Since any two fibers $F$ over points of $\PP^1$ are linearly equivalent, $\Lambda$ is base point free.

 Suppose that $\gamma$ is a curve of $X$ which is not contained in $V_0$. If $\pi(\gamma)=\PP^1$ then $(\gamma\cdot F)>0$ and $(\gamma\cdot V_0)\ge 0$ so that $\gamma$ is not contracted by $\Lambda$. If $\gamma$ is in a fiber of $F$ then $(\gamma\cdot F)=0$. Let $F\cong \PP^r$ be the fiber of $\pi$ containing $\gamma$. Let $h=F\cdot V_0$, a hyperplane section of $F$. Then
 $(\gamma\cdot V_0)=(\gamma\cdot h)_{F}>0$. Thus $\gamma$ is not contracted by $\Lambda$. By induction on $r$, we have that $C$ is the only curve on $V_0$ which is contracted by $\Lambda$. We have thus proven the induction statement for $r$.
 
 \end{proof}

\begin{Example}\label{Example3} Let $k$ be an algebraically closed field.
Suppose that $s$ is a positive integer and $a_i\in \ZZ_+$ are positive integers for $1\le i\le s$. Suppose that $d>1$.
Then there exists a connected reduced projective scheme $X$ over $k$ which is equidimensional of dimension $d$ with a line bundle $\mathcal L$ on $X$
and a bounded function $\sigma(n)$ such that 
$$
\dim_k\Gamma(X,\mathcal L^n)=\lambda(n)\binom{d+n-1}{d-1}+\sigma(n),
$$
where $\lambda(n)$ is the periodic function
$$
\lambda(n)=|\{i\mid n\equiv 0 (a_i)\}|.
$$
The Kodaira-Iitaka dimension of $L$ is  $q(L)=d-1$. Let $m'=\mbox{LCM}\{a_i\}$.
The limit
$$
\lim_{n\rightarrow \infty}\frac{\dim_k L_n}{n^{d-1}}
$$
exists whenever $n$ is constrained to be in an arithmetic sequence $a+bm'$ (with any fixed $a$). We  have that $\dim_kL_n\ne 0$ for all $n$ if some $a_i=1$, so the
conclusions of Theorem \ref{Theorem10} do not quite hold in this example.
\end{Example}

\begin{proof} Let $E$ be an elliptic curve over $k$. Let $p_0, p_1,\ldots, p_s$ be points on $E$ such that the line bundles $\mathcal O_E(p_i-p_0)$ have order $a_i$. Let $S=E\times_k\PP^{d-1}_k$, and define line bundles $\mathcal L_i=\mathcal O_E(p_i-p_0)\otimes \mathcal O_{\PP^{d-1}}(1)$ on $S$.
The Segre embedding gives  a closed embedding of $S$ in $\PP^{r}$ with $r=3d-1$ Let $\pi:X=\PP(\mathcal O_{\PP^1}(-1)\bigoplus \mathcal O_{\PP^1}^r)
\rightarrow \PP^1$ be the projective bundle, and let $C$ be the section corresponding to the surjection
of $\mathcal O_{\PP^1}(-1)\bigoplus \mathcal O_{\PP^1}^r$ onto the first factor. Let $b_1,\ldots,b_s$ be distinct points of $\PP^1$ and let $F_i$ be the fiber by $\pi$ over $b_i$. Let $S_i$ be an embedding of $S$ in $F_i$. We can if necessary make a translation of $S_i$ so that the point $c_i=C\cdot F_i$ lies
on $S_i$, but is not contained in $p_j\times\PP^{d-1}$ for any $j$. 
We have a line bundle $\mathcal L'$ on the (disjoint) union $T$ of the $S_i$ defined by $\mathcal L'|S_i=\mathcal L_i$.

By Lemma \ref{Lemma52}, there is a morphism $\phi:X\rightarrow Y$ which only contracts the curve $C$. $\phi$ is actually birational and an isomorphism away from $C$, but we do not need to verify this, as we can certainly obtain this  after  replacing $\phi$ with  the Stein factorization of $\phi$. Let $Z=\phi(T)$. 
The birational morphism $T\rightarrow Z$ is an isomorphism away from the points $c_i$, which are not contained on the support of the divisor defining $\mathcal L'$. Thus $\mathcal L'|(T\setminus\phi(C))$ extends naturally to a line bundle $\mathcal L$ on $Z$.

We have a short exact sequence
$$
0\rightarrow \mathcal O_Z\rightarrow \bigoplus_{i=1}^s \mathcal O_{S_i}\rightarrow \mathcal F\rightarrow 0
$$
where $\mathcal F$ has finite support. Tensoring this sequence with $\mathcal L^n$ and taking global sections, we obtain that
$$
0\le \sum_{i=1}^s\dim_k\Gamma(S_i,\mathcal L_i^n)-\dim_k\Gamma(Z,\mathcal L^n) \le \dim_k\mathcal F
$$
for all $n$. Since 
$$
\Gamma(S_i,\mathcal L_i^n)\cong \Gamma(E,\mathcal O_E(n(p_i-p_0)))\otimes_k\Gamma(\PP^{d-1},\mathcal O_{\PP^{d-1}}(n))
$$
by the Kuenneth formula, we obtain the conclusions of the example.

 \end{proof}

\section{Perversity on nonreduced schemes}

In this section we give a series of examples, showing interesting growth of graded linear series on nonreduced projective schemes. The examples show that  the wonderful theorems about growth for graded linear series on varieties do not generalize to nonreduced schemes.

\subsection{An example with maximal Kodaira-Iitaka dimension} 

Let $i_1=2$ and $r_1=\frac{i_1}{2}$. For $j\ge 1$, inductively define $i_{j+1}$ so that $i_{j+1}$ is even and $i_{j+1}>2^ji_j$.
Let $r_{j+1}=\frac{i_{j+1}}{2}$. For $n\in \ZZ_+$, define

\begin{equation}\label{eqsigma}
\sigma(n)=\left\{\begin{array}{ll}
1&\mbox{ if } n=1\\
\frac{i_j}{2} &\mbox{ if }i_j\le n<i_{j+1}
\end{array}\right.
\end{equation}

 \begin{Lemma}\label{Lemma1}
 Suppose that $a\in \NN$ and $r\in \ZZ_+$. Then given $m>0$ and $\epsilon>0$, there exists a positive integer $n=a+br$ with $b\in \NN$ such that $n>m$ and
 $$
 \left| \frac{\sigma(n)}{n}-\frac{1}{2}\right|<\epsilon
 $$
 \end{Lemma}
 
 \begin{proof} Choose $j$ sufficiently large that $i_j>m$, $i_j+r<i_{j+1}$  and
 \begin{equation}\label{eq1}
 \frac{i_j}{2(i_j+k)}> \frac{1}{2}-\epsilon
 \end{equation}
 for $0\le k<r$.
There exists $n=i_j+k$ with $0\le k<r$ in the arithmetic sequence $a+br$.
$$
\frac{\sigma(n)}{n}=\frac{i_j}{2n}=\frac{i_j}{2(i_j+k)}.
$$
By (\ref{eq1}),
$$
\frac{1}{2}\ge \frac{i_j}{2(i_j+k)}> \frac{1}{2}-\epsilon.
$$
\end{proof}

\begin{Lemma}\label{Lemma2}
 Suppose that $a\in \NN$ and $r\in \ZZ_+$. Then given $m>0$ and $\epsilon>0$, there exists a positive integer $n=a+br$ with $b\in \NN$ such that $n>m$ and
 $$
 \left| \frac{\sigma(n)}{n}\right|<\epsilon.
 $$
 \end{Lemma}

\begin{proof} Choose $j$ sufficiently large that $i_j>m+r$, $2^ji_j>r$ and
\begin{equation}\label{eq2}
\frac{i_j}{2(2^ji_j-k)}<\epsilon
\end{equation}
for $0<k\le r$. Let $n=i_{j+1}-k$ with $0<k\le r$ in the arithmetic sequence $a+br$.
$$
\frac{\sigma(n)}{n}=\frac{i_j}{2n}=\frac{i_j}{2(i_{j+1}-k)}.
$$
By (\ref{eq2}), 
$$
0<\frac{i_j}{2(i_{j+1}-k)}<\epsilon.
$$
\end{proof}

It follows from the previous two lemmas that 
$$
\lim_{n\rightarrow \infty} \frac{\sigma(n)}{n}
$$
does not exist, even when $n$ is constrained to lie in an  arithmetic sequence.

\begin{Example}\label{Example25} Let $k$ be a field, and let $X$ be the $d$-dimensional projective nonreduced  $k$-scheme consisting of a double linear hyperplane in $\PP^{d+1}_k$, with $d\ge 1$.
 There exists a graded  linear series
$L$  for $\mathcal O_X(2)$ with maximal Kodaira-Iitaka dimension $q(L)=d$, such that 
\begin{equation}\label{eqKI6}
\lim_{n\rightarrow \infty} \frac{\dim_k(L_n)}{n^d}
\end{equation}
does not exist, even when $n$ is constrained to lie in an arithmetic sequence.
\end{Example}

$L_n$ is constructed to be a subspace of $\Gamma(X,\mathcal O_X(2n))$ which generates 
$\mathcal O_X(2n)$ at all points except along a fixed $d-1$ dimensional linear subspace of $X$.

\begin{proof}  We can choose homogeneous coordinates on $\PP^2_k$ so that $X=\mbox{Proj}(S)$ where $S=k[x_0,x_1,\ldots,x_{d+1}]/(x_1^2))$,
and the fixed linear subspace is $Z(x_0,x_1)\subset X$. Let $\overline x_i$ be the classes of $x_i$ in $S$, so that $S=k[\overline x_0,\overline x_1,\ldots,\overline x_{d+1}]$ (with $\overline x_1^2=0$). Define a graded family of homogeneous ideals in $S$ by
$$
J_n=(\overline x_0^n, \overline x_1\overline x_0^{n-\sigma(n)})
$$
for $n\ge 1$ and $J_0=S$. We have that $J_mJ_n\subset J_{m+n}$ since $\sigma(j)\le \sigma(k)$ for $j\le k$.

Let $\tilde J_n$ be the sheafification of $J_n$ on $X$. Define
$$
L_n=\Gamma(X,\tilde J_n\otimes \mathcal O_X(2n))\subset \Gamma(X,\mathcal O_X(2n))=S_{2n}.
$$
The $L_n$ define a graded linear series $L$ on $X$.

Let $M_i$ be the set of all monomials in the $d+1$ variables $\overline x_0, \overline x_2,\ldots,\overline x_{d+1}$ of degree $i$.

$L_n$ has a $k$-basis consisting of $\overline x_0^{n}M_{n-\alpha}$  and
$\overline x_1\overline x_0^{n-\sigma(n)}M_{n+\sigma(n)-1}$.
Thus
\begin{equation}\label{eq60}
\dim_kL_n=\binom{d+n}{d}+\binom{d+n+\sigma(n)-1}{d}.
\end{equation}

Let $P(t)$ be the degree $d$ rational polynomial $P(t)=\binom{d+t}{d}$. We have that
\begin{equation}\label{eq90}
P(t)=\frac{t^d}{d!}+\mbox{ lower order terms in $t$}.
\end{equation}

Let $n=a+br$ be an arithmetic sequence (with $a,r$ fixed). Suppose that $n_0$ is a positive integer and $\epsilon>0$ is a real number. 

Since $0\le \sigma(n)\le\frac{n}{2}$ for all $n$, it follows from (\ref{eq60}) and (\ref{eq90}) that there exists $n_1\ge n_0$ such that $n\ge n_1$ implies
$$
|\frac{\dim_kL_n}{n^d}-\frac{1}{d!}(1+(1+\frac{\sigma(n)}{n})^d)|<\frac{\epsilon}{2}.
$$
By Lemma \ref{Lemma1}, there exists an integer $n_2>n_1$ such that $n_2$ is in the arithmetic sequence $n=a+br$ and
$$
|\frac{1}{d!}(1+(1+\frac{\sigma(n_2)}{n_2})^d)-\frac{1+(1+\frac{1}{2})^d}{d!}|<\frac{\epsilon}{2}.
$$
Thus 
\begin{equation}\label{eq63}
|\frac{\dim_kL_{n_2}}{n_2^d}-\frac{1+(\frac{3}{2})^d}{d!}|<\epsilon.
\end{equation}
On the other hand, by Lemma \ref{Lemma2}, there exists an integer $n_3>n_1$ such that $n_3$ is in the arithmetic sequence $n=a+br$ and 
$$
|\frac{1}{d!}(1+(1+\frac{\sigma(n_3)}{n_3})^d)-\frac{2}{d!}|<\frac{\epsilon}{2}.
$$
Thus 
\begin{equation}\label{eq64}
|\frac{\dim_kL_{n_3}}{n_3^d}-\frac{2}{d!}|<\epsilon.
\end{equation}
By (\ref{eq63}) and (\ref{eq64}) we have that the limit (\ref{eqKI6}) does not exist when $n$ is constrained to line in the arithmetic sequence $n=a+br$.
Since this sequence was arbitrary, we have obtained the conclusions of the example.
\end{proof}

The above example is a graded linear series on a double linear hyperplane $X$ in $\PP^{d+1}$. We observe that if $\mathcal L$ is a line bundle on $X$, then not only does the limit exist for the section ring of $\mathcal L$, it is even a rational number. Suppose that $\mathcal L$ is a line bundle on $X$. Let $X_0$ be the reduced linear subspace of $\PP^{d+1}$ which has the same support as $X$. We have a short exact sequence
of coherent $\mathcal O_{X_0}$ modules
$$
0\rightarrow \mathcal O_{X_0}(-1)\rightarrow \mathcal O_X\rightarrow \mathcal O_{X_0}\rightarrow 0.
$$
From the exact sequence 
$$
H^1(X,\mathcal O_{X_0}(-1))\rightarrow \mbox{Pic}(X)\rightarrow \mbox{Pic}(X_0)\rightarrow H^2(X,\mathcal O_{X_0}(-1))
$$
of Exercise III.4.6 \cite{H}, and from the cohomology of projective space, we see that restriction gives an isomorphism $\mbox{Pic}(X)\cong\mbox{Pic}(X_0)$.
Since every line bundle on the linear subspace $X_0$ is the restriction of a line bundle on $\PP^{d+1}$, we have that $\mathcal L\cong\mathcal O_{\PP^{d+1}}(a)\otimes \mathcal O_X$ for some $a\in\ZZ$.
From the exact sequences 
$$
0\rightarrow \mathcal O_{\PP^{d+1}}(na-2)\rightarrow \mathcal O_{\PP^{d+1}}(na)\rightarrow \mathcal L^n\rightarrow 0,
$$
we see that 
$$
\dim_k\Gamma(X,\mathcal L^n)=\dim_k\Gamma(\PP^{d+1},\mathcal O_{\PP^{d+1}}(na))-\dim_k\Gamma(\PP^{d+1},\mathcal O_{\PP^{d+1}}(na-2)),
$$
which is zero for all positive $n$ if $a<0$, is 1 for all positive $n$ if $a=0$, and is equal to
the polynomial
$$
\binom{na+(d+1)}{d+1}-\binom{na-2+(d+1)}{d+1}
$$
for  $n\ge 2$ if $a>0$.

\subsection{Examples with Kodaira-Iitaka dimension $-\infty$}
It is much easier to construct perverse examples with Kodaira-Iitaka dimension $-\infty$, since the condition $L_mL_n\subset L_{m+n}$ can be trivial
in this case. If $X$ is a reduced variety, and $L$ is a graded linear series on $X$, then it follows from Corollary \ref{Cor72} that there is an upper bound 
$\dim_kL_n<\beta n^{q(L)}$
for all $n$. However, for nonreduced varieties of dimension $d$, we only have the upper bound $\dim_k<\gamma n^d$ of (\ref{eqKI4}).
Here is an example with $q(L)=-\infty$ and maximal growth of order $n^d$.

\begin{Example}\label{Example2} Let $k$ be a field, and let $X$ be the one dimensional projective non reduced  $k$-scheme consisting of a double line in $\PP^2_k$. Let $T$ be a subset of the positive integers.
 There exists a graded  linear series
$L$  for $\mathcal O_X(2)$ such that 
$$
\dim_kL_n= \left\{\begin{array}{ll}
n+1&\mbox{ if }n\in T\\
0&\mbox{ if }n\not\in T
\end{array}\right.
$$

In the example, 
we have that $q(L)=-\infty$, but $\dim_kL_n$ is $O(n)=O(n^{\dim X})$. 
\end{Example}

\begin{proof}
 We can choose homogeneous coordinates coordinates on $\PP^2_k$ so that $X=\mbox{Proj}(S)$, where $S=k[x_0,x_1,x_2]/(x_1^2)$. Let $\overline x_i$ be the classes of $x_i$ in $S$, so that $S=k[\overline x_0,\overline x_1,\overline x_2]$.
Define a graded linear series $L$ for $\mathcal O_X(2)$ by
defining $L_n$ to be the $k$-subspace of $\Gamma(X,\mathcal O_X(2n))$ spanned by $\{\overline x_1\overline x_0^i\overline x_2^j\mid i+j=n\}$
if $n\in T$ and $L_n=0$ if $n\not\in T$. Then
$$
\dim_kL_n= \left\{\begin{array}{ll}
n+1&\mbox{ if }n\in T\\
0&\mbox{ if }n\not\in T
\end{array}\right.
$$
\end{proof}

We modify the above example a little bit to find another  example with interesting growth.

\begin{Theorem}\label{Theorem21} Let $k$ be a field, and let $X$ be the one dimensional projective non reduced  $k$-scheme consisting of a double line in $\PP^2_k$. Let $T$ be any infinite subset of the positive integers $\ZZ_+$ such that $\ZZ_+\setminus T$ is also infinite.
 There exists a graded  linear series
$L$  for $\mathcal O_X(2)$ such that 
$$
\dim_kL_n= \left\{\begin{array}{ll}
\lceil \log(n)\rceil+1 &\mbox{ if }n\in T\\
\lceil \frac{\log(n)}{2}\rceil+1&\mbox{ if }n\not\in T
\end{array}\right.
$$
In this example we have $q(L)=\infty$.
\end{Theorem}

\begin{proof}
 We can choose homogeneous coordinates coordinates on $\PP^2_k$ so that $X=\mbox{Proj}(S)$, where $S=k[x_0,x_1,x_2]/(x_1^2)$. Let $\overline x_i$ be the classes of $x_i$ in $S$, so that $S=k[\overline x_0,\overline x_1,\overline x_2]$.
Define
$$
\lambda(n)=\left\{\begin{array}{ll}
\lceil \log(n)\rceil &\mbox{ if }n\in T\\
\lceil \frac{\log(n)}{2}\rceil&\mbox{ if }n\in \ZZ_+\setminus T.
\end{array}\right.
$$ 
Define a graded linear series $L$ for $\mathcal O_X(2)$ by
defining $L_n$ to be the $k$-subspace of $\Gamma(X,\mathcal O_X(n))$ spanned by 
$$
\overline x_0^{n}\overline x_1,\overline x_0^{n-1}\overline x_1\overline x_2,\ldots, \overline x_0^{n-\lambda(n)}\overline x_1\overline x_2^{\lambda(n)}.
$$
Then $L_n$ has the desired property.
\end{proof}

The following is an example of a line bundle on a non reduced scheme for which there is interesting growth. The characteristic $p>0$ plays a role in the construction.

\begin{Example}\label{Example16} Suppose that  $d\ge 1$. There exists an irreducible but  nonreduced projective variety $Z$ of dimension $d$ over a field of  positive characteristic $p$,   and a line bundle $\mathcal N$ on $Z$, whose Kodaira-Iitaka dimension is  $-\infty$, such that 
$$
\dim_k\Gamma(Z,\mathcal N^{n})=
\left\{\begin{array}{ll}
 \binom{d+n-1}{d-1}&\mbox{ if $n$ is  a power of $p$}\\
 \\
0&\mbox{ otherwise}
\end{array}\right.
$$

\end{Example}
In particular,  given a positive integer $r$, there exists at least one integer   $a$ with $0\le a<r$ such that the limit
$$
\lim_{n\rightarrow \infty}\frac{\dim_k\Gamma(Z,\mathcal N^n)}{n^{d-1}}
$$
does not exist when $n$ is constrained to lie   in the arithmetic sequence $a+br$. 

\begin{proof}

Suppose that $p$ is a prime number such that $p\equiv 2\,\, (3)$.
In Section 6 of \cite{CS}, a projective genus 2 curve $C$ over an algebraic function field $k$ of characteristic $p$ is constructed, which has a $k$-rational point $Q$ and a degree zero line bundle $\mathcal L$ with the properties that
$$
\dim_k \Gamma(C,\mathcal L^n\otimes \mathcal O_C(Q))
=\left\{\begin{array}{ll}
1&\mbox{ if $n$ is  a power of $p$}\\
0&\mbox{ otherwise}
\end{array}\right.
$$
and
\begin{equation}\label{eq32}
\Gamma(C,\mathcal L^n)=0\mbox{ for all }n.
\end{equation}

Let $\mathcal E=\mathcal O_C(Q)\bigoplus \mathcal O_C$. Let $S=\PP(\mathcal E)$ with natural projection $\pi:S\rightarrow C$, a ruled surface over $C$. Let $C_0$ be the section of $\pi$ corresponding to the
 surjection onto the second factor $\mathcal E\rightarrow \mathcal O_C\rightarrow 0$. By Proposition V.2.6 \cite{H}, we have that
$\mathcal O_S(-C_0)\otimes_{\mathcal O_S}\mathcal O_{C_0}\cong \mathcal O_C(Q)$. Let $X$ be the nonreduced subscheme $2C_0$ of $S$. We have a short exact sequence
$$
0\rightarrow \mathcal O_{C}(Q)\rightarrow \mathcal O_{X}\rightarrow \mathcal O_C\rightarrow 0.
$$
 Let $\mathcal M=\pi^*(\mathcal L)\otimes_{\mathcal O_S}\mathcal O_X$. Then
we have short exact sequences
\begin{equation}\label{eq31}
0\rightarrow \mathcal L^n\otimes_{\mathcal O_C}\mathcal O_C(Q)\rightarrow \mathcal M^n\rightarrow \mathcal L^n\rightarrow  0.
\end{equation}

 By (\ref{eq31}) and (\ref{eq32}), we have that

$$
\begin{array}{lll}
\dim_k\Gamma(X,\mathcal M^n)&=&\dim_k\Gamma(C,\mathcal L^n\otimes \mathcal O_C(Q))\\
\\
&=&\left\{
\begin{array}{ll}
1&\mbox{ if $n$ is  a power of $p$}\\
0&\mbox{ otherwise}
\end{array}\right.\\
\end{array}
$$

Now let $Z=X\times \PP_k^{d-1}$ and $\mathcal N=\mathcal M\otimes \mathcal O_{\PP}(1)$. By the Kuenneth formula, we have that
$$
\Gamma(Z,\mathcal N^n)=\Gamma(X,\mathcal M^n)\otimes_k\Gamma(\PP^{d-1},\mathcal O_{\PP}(n))
$$
from which the conclusions of the example follow.
\end{proof}

\section{A graded family of $m$-primary ideals which does not have a limit}

In this section we give an example showing lack of limits for families of $m$-primary ideals in non reduced rings. 
A family of ideals $\{I_n\}$ in a $d$-dimensional local ring $R$ indexed by $n$ is a family of ideals in $R$ if $I_0=R$ and $I_mI_n\subset I_{m+n}$ for all $m,n$.  $\{I_n\}$ is an $m_R$-primary family if $I_n$ is $m_R$-primary for $n\ge 1$.
The limit 
 \begin{equation}\label{eq100}
 \lim_{n\rightarrow \infty} \frac{\ell(R/I_n)}{n^d}
 \end{equation}
is 
 shown to exist in many cases in papers of 
 Ein, Lazarsfeld and Smith \cite{ELS}, Mustata \cite{Mus}, Lazarsfeld and Mustata \cite{LM} and of the author \cite{C}. It is shown in Theorem 5.9 \cite{C} that the limit (\ref{eq100})
  always exists if $\{I_n\}$ is a graded family of $m_R$-primary ideals in a $d$-dimensional unramified equicharacteristic local ring $R$ with perfect residue field. 

\begin{Example}\label{Example1} Let $k$ be a field and $R$ be the non reduced $d$-dimensional local ring $R=k[[x_1,\ldots,x_d,y]]/(y^2)$.
There exists a  graded family of $m_R$-primary ideals $\{I_n\}$ in $R$ such that the limit
$$
\lim_{n\rightarrow \infty} \frac{\ell(R/I_n)}{n^d}
$$
does not exist, even when $n$ is constrained to lie in an arithmetic sequence.
Here $\ell$ denotes the length of an $R$-module.
\end{Example}

\begin{proof}  Let $\overline x_1, \ldots,\overline x_d,\overline y$ be the classes of $x_1,\ldots,x_d,y$ in $R$.
Let $N_i$ be the set of monomials of degree $i$ in the variables $\overline x_1,\ldots,\overline x_d$.
Let $\sigma(n)$ be the function defined in (\ref{eqsigma}).
Define $M_R$-primary ideals $I_n$ in $R$ by
 $I_n=(N_n,\overline yN_{n-\sigma(n)})$ for $n\ge 1$ (and $I_0=R$). 

We first verify that $\{I_n\}$ is a graded family of ideals, by showing that $I_mI_n\subset I_{m+n}$ for all $m,n>0$.
This follows since
$$
I_mI_n=(N_{m+n},\overline yN_{(m+n)-\sigma(m)}, \overline yN_{(m+n)-\sigma(n)})
$$
and  $\sigma(j)\le \sigma(k)$ for $k\ge j$.

$R/I_n$ has a $k$-basis consisting of 
$$
\{N_i|i<n\}\mbox{ and }\{\overline yN_j|j<n-\sigma(n)\}.
$$
Thus 
$$
\ell(R/I_n)=\binom{n}{d}+\binom{n-\sigma(n)}{d}.
$$
By an similar argument to that of the proof of Example \ref{Example2}, we obtain the conclusions of this example.
\end{proof}

\section{Appendix}
In this appendix, we give a proof of Lemma \ref{LemmaKI}  stated in Section \ref{Section2}. We begin with a proof of  another lemma we will need.

\begin{Lemma} Suppose that $L$ is a graded linear series on  a projective  scheme $X$ over a field $k$, and 
 $L$ is a finitely generated $L_0$-algebra. Then 
\begin{equation}\label{eqKI7}
q(L)=\left\{\begin{array}{ll}
{\rm Krull\,\, dimension}\,(L)-1&\mbox{ if }{\rm Krull\,\, dimension}\,(L)>0\\
-\infty&\mbox{ if }{\rm Krull\,\, dimension}\,(L)=0.
\end{array}\right.
\end{equation}
\end{Lemma} 

\begin{proof}
In the case when $L_0=k$, the lemma  follows from  graded Noether normalization (Theorem 1.5.17 \cite{BH}). 
For a general graded linear series $L$, we always have that $k\subset L_0\subset \Gamma(X,\mathcal O_X)$, which is a finite dimensional $k$-vector space since $X$ is a projective $k$-scheme. $L$ is thus a finitely generated $k$ algebra. Let $m=\sigma(L)$ and $y_1,\ldots, y_m\in L$ be homogeneous elements of positive degree which are algebraically independent over $k$. Extend to homogeneous elements of positive degree $y_1,\ldots, y_n$ which generate $L$ as an $L_0$-algebra. 
Let $B=k[y_1,\ldots, y_n]$. We have that $\sigma(L)\le \sigma(B)\le \sigma(L)$ so $\sigma(B)=\sigma(L)$. By the first case ($L_0=k$) proven above, we have that 
$q(B)=\mbox{Krull dimension}(B)-1$ if $q(B)\ge 0$ and $q(B)=-\infty$ if $\mbox{Krull dimension}(B)=0$. 
 Since $L$ is finite over $B$, we have that $\mbox{Krull dimension}(L)=\mbox{Krull dimension}(B)$. Thus the lemma  holds.
\end{proof}

\subsection{Proof of Lemma \ref{LemmaKI}}

We will first establish the formulas (\ref{eqKI1}) and (\ref{eqKI4}).
Since $X$ is projective over $k$, we have an expression $X=\mbox{Proj}(A)$ where $A$ is the quotient of a standard graded polynomial ring 
$R=k[x_0,\ldots,x_n]$ by a homogeneous ideal $I$, which we can take to be saturated; that is, $(x_0,\ldots,x_n)$ is not an associated prime of $I$.
Let $\mathfrak p_1,\ldots,\mathfrak p_t$ be the associated primes of $I$. By graded prime avoidance (Lemma 1.5.10 \cite{BH}) there exists a form $F$ in $k[x_0,\ldots,x_n]$ of
some positive degree $c$ such that $F\not\in \cup_{i=1}^t\mathfrak p_i$. Then $F$ is a nonzero divisor on $A$, so that
$A\stackrel{F}{\rightarrow}A(c)$ is 1-1. Sheafifying, we have an injection
\begin{equation}\label{eqKI5}
0\rightarrow \mathcal O_X\rightarrow \mathcal O_X(c).
\end{equation}

Since $\mathcal O_X(c)$ is ample on $X$, there exists $f>0$ such that $\mathcal A := \mathcal L\otimes \mathcal O_X(cf)$ is ample.
From (\ref{eqKI5}) we then have
 a 1-1 $\mathcal O_X$-module homomorphism $\mathcal O_X\rightarrow \mathcal O_X(cf)$, and a 1-1 $\mathcal O_X$-module homomorphism
$\mathcal L\rightarrow \mathcal A$, which induces  inclusions of graded $k$-algebras
$$
L\subset \bigoplus_{n\ge 0}\Gamma(X,\mathcal L^n)\subset B:= \bigoplus_{n\ge 0}\Gamma(X,\mathcal A^n).
$$
There exists a positive integer $e$ such that $\mathcal A^e$ is very ample on $X$. Thus, by Theorem II.5.19 and Exercise II.9.9 \cite{H}, $B'=\bigoplus_{n\ge 0}B_{en}$ is finite over a coordinate ring $S$ of $X$
and  
\begin{equation}\label{eqKI9}
\dim_k\Gamma(X,\mathcal A^{en})=P_S(n)
\end{equation}
for $n\gg 0$ where $P_S(n)$ is the Hilbert Polynomial of $S$.

Now $B$ is a finitely generated module over $B'$. Thus
$$
\mbox{Krull dimension}(B)=\mbox{Krull dimension}(B')=\mbox{Krull dimension}(S)=\dim X+1=d+1.
$$
Further, $B$ is a finitely generated $k$-algebra so that $q(B)=\dim(X)$ by (\ref{eqKI7}).
Thus we have obtained formula (\ref{eqKI1}), 
$$
q(L)\le \dim X.
$$
$\mathcal A$ is ample on $X$, so that
$$
\dim_k\Gamma(X,\mathcal A^n)=\chi(\mathcal A^n)
$$
for $n\gg 0$. The Euler characteristic $\chi(\mathcal A^n)$ is 
a polynomial in $n$  for $n\gg 0$ by Proposition 8.8a \cite{I}, which necessarily has degree $d$ by (\ref{eqKI9}),
so  there exists a positive constant $\gamma$ such that 
$$
\dim_k L_n<\gamma n^d
$$
for all $n$, giving us formula (\ref{eqKI4}).

We will now establish formula (\ref{eqKI2}).
Suppose that $q(L)\ge 0$.
Let $L^i$ be the $k$-subalgebra of $L$ generated by $L_j$ for $j\le i$. For $i$ sufficiently large, we have that $q(L^i)=q(L)$. For such an $i$, 
since $L^i$ is a finitely generated $L_0$-algebra, we have that there exists a number $e$ such that the Veronese algebra $L^*$ defined by
$L^*_n=(L^i)_{en}$ is generated as a $L_0$-algebra in degree 1. Thus, since $L_0$ is an Artin ring,  and $L^*$ has Krull dimension $q(L)+1$ by (\ref{eqKI7}), $L^*$ has a Hilbert polynomial $P(t)$ of degree $q(L)$, satisfying  
$\ell_{L_0}(L^*_n)=P(n)$ for $n\gg 0$ (Corollary to Theorem 13.2 \cite{Ma2}), where $\ell_{L_0}$ denotes length of an $L_0$ module, and thus
$\dim_kL^*_n=(\dim_kL_0)P(n)$ for $n\gg 0$. Thus there exists a positive constant $\alpha$ such that $\dim_kL^*_n >\alpha n^{q(L)}$ for all $n$, and so
$$
\dim_kL_{en}>\alpha n^{q(L)}
$$
for all positive integers $n$, which is formula (\ref{eqKI2}).

Finally, we will establish the fourth statement of the Lemma.
Suppose that $X$ is reduced and $0\ne L_n$ for some $n>0$. Consider the graded $k$-algebra homomorphism
$\phi:k[t]\rightarrow L$ defined by $\phi(t)=z$ where $k[t]$ is graded by giving $t$ the weight $n$. The kernel of $\phi$ is weighted homogeneous, so it is either 0 or 
$(t^s)$ for some $s>1$.  Thus if $\phi$ is not 1-1 then there exists $s>1$ such that $z^s=0$ in $L_{ns}$. We will show that this cannot happen.
Since $z$ is a nonzero global section of $\Gamma(X,\mathcal L^n)$, there exists $Q\in X$ such that the image of $z$ in $\mathcal L^n_Q$ is $\sigma f$ where $f\in \mathcal O_{X,Q}$ is nonzero and $\sigma$ is a local generator of $\mathcal L^n_Q$.
 The image of $z^s$ in $\mathcal L^{sn}_Q=\sigma^s\mathcal O_{X,Q}$ is $\sigma^sf^s$. We have that $f^s\ne 0$ since
$\mathcal O_{X,Q}$ is reduced. Thus $z^s\ne 0$. We thus have that $\phi$ is 1-1, so $q(L)\ge 0$.

\end{document}